\documentclass[final,1p,times]{elsarticle}
\usepackage{amssymb}
\usepackage{amsmath}
\usepackage{amsthm}
\usepackage[utf8]{inputenc}
\usepackage{marginnote}
\usepackage{amsfonts}
\usepackage{textcomp}
\usepackage{color,soul}
\usepackage{placeins}
\usepackage{tikz}
\usepackage{makecell}
\usepackage{amssymb}
\usepackage{graphicx}
\usepackage{hyperref}
\usepackage[none]{hyphenat}
\usepackage{lipsum}
\newtheorem{thm}{Theorem}[section]

\newtheorem{rem}[thm]{Remark}
\newtheorem{prop}[thm]{Proposition}
\newtheorem{cor}[thm]{Corollary}

\newdefinition{defn}[thm]{Definition}
\allowdisplaybreaks

\begin{document}
\begin{frontmatter}
		\title{ Inequalities for Pairs of Measure Spaces and Applications }
        \author{P. D. Johnson}
		\ead{johnspd@auburn.edu}
		\author{R. N. Mohapatra}
		\ead{ram.mohapatra@ucf.edu}
		\author{Shankhadeep Mondal}
		\ead{shankhadeep.mondal@ucf.edu}

		\address{Department of Mathematics and Statistics, Auburn University, Auburn, AL-36849 }
       \address{School of Mathematics, University of Central Florida, Orlando, Florida-32816}
        \address{School of Mathematics, University of Central Florida, Orlando, Florida-32816}
		\cortext[ram]{Corresponding Author:ram.mohapatra@ucf.edu}
		
\begin{abstract}
We study a family of inequalities on pairs of measure spaces involving functions defined on product domains. Our main result establishes a Jensen-type inequality under a general product-measure framework, extending classical inequalities such as H\"older’s and Minkowski’s as special cases. The inequality admits sharp characterizations of equality
and yields quantitative, variational, and probabilistic refinements under additional convexity assumptions. Several corollaries illustrate power-mean, entropy-type, and erasure-robust inequalities, as well as applications to convolution-type operators and weighted discrete models.
\end{abstract}


\end{frontmatter}

\section{Introduction and background}

Suppose that $k$ is a positive integer.  A {\em $k$-uniform hypergraph} is a pair $(V, E)$ in which $V$ is a set (of {\em vertices}) and $E$ is a multiset (of {\em hyperedges}, or just {\em edges}, for short) of multisubsets of $V$, each with multiset cardinality $k$.

To call $E$ a multiset means that ``repeated edges'' are allowed.  To call each $e \in E$ a multisubset of $V$ means that elements of $V$ may appear more than once as elements of $e$.  For $e \in E$ and $v \in V$ let $M(v, e)$ denote the number of times $v$ appears in, or ``on'', the edge $e$.  $M(v, e)$ is sometimes called the {\em multiplicity} of $v$ on $e$.  The array $M = [M(v, e); v \in V, e \in E]$ is sometimes called the vertex-edge incidence matrix of the hypergraph.  The requirement of $k$-uniformity may be restated: for each $e \in E$, $\sum _{v \in V} M(v, e) = k$; in other words, the column sums of $M$ are all equal to $k$.  Throughout, we rule out the presence of {\em isolated vertices} $v \in V$ such that $M(v, e) = 0$ for all $e \in E$.

If $G = (V, E)$ is a hypergraph (not necessarily $k$-uniform for some $k$) with vertex-edge incidence matrix $M$  (a $p \times q$ matrix, $p = |V|, q = |E|$), the {\em degree} (or {\em valence}, of a vertex $v \in V$, denoted $d(v)$, is the number of appearances of $v$ on edges of the hypergraphs: $d(v) = \sum _{e \in E} M(v, e)$. If the degrees of the vertices of $G$ are all equal, $G$ is {\em regular}.  The {\em average degree} in $G$ is, not surprisingly, 

$$\begin{array}{ll}
\bar {d} & = \frac {1}{|V|} \sum _{v \in V} d(v) = \frac {1}{p} \sum _{v \in V} \sum _{e \in E} M(v, e)\\ [.2in]
& = \frac {1}{p} \sum _{e \in E} \sum _{v \in V} M(v, e).
\end{array}$$

If $G$ happens to be $k$-uniform, then

$$
\bar {d} = \frac {1}{p} |E| k = \frac {kq}{p}.
$$
Since a hypegraph is completely described by its vertex-edge incidence matrix, the following theorem, which does not refer to hypergraphs, actually does provide a uncountable family of inequalities relating degrees of adjacent vertices (vertices that are together on at least one edge) to the average degree.

\begin{thm} [\cite {JP}]  \label {t1}
Suppose $M = [m_{ij}]$ is a $p \times q$ matrix of real numbers with constant column sum $c$ and row sums $d_1, \dots, d_p$.  Suppose that $\varphi$ is a function such that $f(x) = x \varphi (x)$ is convex on an interval containing $d_1, \dots, d_p$.
Let $\bar {d} = \frac {1}{p} \sum _{i=1}^p d_i (= \frac {qc}{p})$.  Then 

$$
\frac {1}{q} \sum _{j=1}^q \sum _{i=1}^p m_{ij} \varphi (d_i) \geq c \varphi (\bar {d}).
$$
If $f$ is strictly convex then equality holds if and only if $d_1 = \cdots = d_p$.
\end{thm}

For definitions of and fundamental results about convexity and strict convexity, see \cite {HL}, pp. 70-74, or almost any text on real analysis.

The proof of Theorem \ref {t1} is an easy exercise:  start on the left hand side of the proposed inequality, change the order of summation, note that $d_i = \sum _{j=1}^q m_{ij}$, and use Jensen's Inequality \cite {HL}, permitted by the convexity of $f$.  

Interest in this result first arises from its application to uniform hypergraphs.  If $M$ is the vertex-edge incidence matrix of a $k$-uniform hypergraph $G = (V, E)$, by which assumption the constant column sum in $M$ is $k$, then Theorem \ref {t1} says that, for a function $\varphi$ satisfying the hypothesis of the theorem, the arithmetic mean, over the edges, of the sums of the values of $\varphi$ at the degrees of the vertices of $G$ appearing on the edges (with $\varphi (d_i)$ appearing in the sum for the $j^{th}$ edge $m_{ij}$ times, the number of times the $i^{th}$ vertex appears on that edge) is greater than or equal to $\varphi (\bar {d})$; and if $f(t) = t \varphi (t)$ is strictly convex on an interval containing the vertex degrees, then equality holds if and only if the hypergraph is regular.

For some choices of $\varphi$, the resulting inequality is intriguing.  For instance, let's take $\varphi (t) = lnt$.  If $f(t) = t lnt$, then $f'' (t) = \frac {1}{t} > 0$ if $t > 0$, so $f$ is strictly convex on $(0, \infty)$.  Plugging $\varphi = ln$ into the inequality of Theorem \ref {t1}, with $k$ replacing $c$, then thrashing around a bit, including exponentiating both sides, we obtain

$$
[ \prod _{j=1}^q ( \prod _{i=1}^p d_i^{m_{ij}})^{1/k} ]^{1/q} \geq \bar {d},
$$
with equality if and only if $d_1, \dots, d_p$ are equal.  For a $k$-uniform hypergraph, this says that the geometric mean, over the edges, of the geometric means of the degrees of the vertices on the edges, is greater than equal the arithmetic mean of the vertex degrees, with equality if and only the hypergraph is regular.

If you fool around with $\varphi (t) = t^s$, $s > 0$, you will obtain an inequality which is well known, and holds for any $d_1, \dots, d_p > 0$ and any $s > 0$:

$$
\frac {1}{p} \sum _{i=1}^p d_i^{s+1} \geq (\frac {1}{p} \sum _{i=1}^p d_i)^{s + 1}.
$$
This is a straightforward consequence of the convexity of $f(x) = x^{s+1}$ on $[0, \infty)$; so the slightly more complicated proof afforded by the proof of Theorem \ref {t1} is of no interest.  Also, this inequality is not interestingly interpretable as a statement about uniform hypergraphs, as was the inequality arising from $\varphi (x) = lnx$.  We are noting all this by way of admitting that of the uncountable infinity of inequalities implied by Theorem \ref {t1}, not all are new and/or interesting; perhaps not many are.

The next advance on this topic was made by the co-authors of this paper in \cite {PJ}.  We realized that Theorem \ref {t1} could be generalized by attaching {\em weights} to the edges of a (uniform) hypergraph, say $wt(e)$ is the weight attached to edge $e$\,and defining $d(v) = \sum _{e \in E(G)} M(v, e) wt(e)$, the {\em weighted degree} of $v$.  [Notes: in an ordinary hypergraph $G = (V, E), wt(e) = 1$ for all $e \in E$.  This is the situation in Theorem \ref {t1}.  Buzzword alert: when $wt(e) \in (0, 1]$ for all $e \in E$, the edges are said to be {\em fuzzy}.]

Theorem \ref {t1} generalizes to Theorem \ref {t2}, just below, which serves to generalize the results in \cite {PJ} about uniform hypergraphs to edge-weighted uniform hypergraphs, in \cite {HL}.  But in this paper we are saying good-bye to hypergraphs:  we are interested in Theorem \ref {t2} for itself, not as a means to use hypergraphs to generate analytic inequalities.  And we intend to generalize Theorem \ref {t2} from being about certain matrices to being about certain products of measure spaces.

\begin{thm} [\cite {PJ}] \label {t2}
Let $M = [m_{ij}]$ be a $p \times q$ matrix of real numbers with constant column sum $c$.  Let $(x_1, \dots, x_q) \in \mathbb {R}^q$ and an interval $I \subseteq \mathbb {R}$ satisfy

\begin{itemize}
\item [(i)] $s = \sum _{i=1} ^q x_i > 0$, and
\item [(ii)] for each $i \in \{1, \dots, p\}$, $d_i = \sum _{j=1}^q m_{ij} x_j \in I$.
\end{itemize}
Let $\bar {d} = \frac {1}{p} \sum _{i=1}^p d_i$.

Suppose that $\varphi$ is a real-valued function on $I$ such that $f$, defined by $f(t) = t \varphi (t)$, is convex on $I$.  Then

$$
s^{-1} \sum _{j=1}^q \sum _{i=1}^p \varphi (d_i) m_{ij} x_j \geq c \varphi (\bar {d});
$$
if $f$ is strictly convex on $I$ then equality holds only if $d_1 = \cdots = d_p$.
\end{thm}

In the application of this result to edge weighted uniform hypergraphs, $c$ is the number of vertices (counting repetitions) on each edge, $p = |V|, q = |E|, M$ is the vertex-edge incidence matrix, $x_1, \dots, x_q$ are the edge weights, $d_1, \dots, d_p$ are the weighted degrees of the vertices, and $\bar {d}$ is the average weighted degree.

Without the connection to edge weighted uniform hypergraphs, unless you are an aficionado of real analytic inequalities, Theorem \ref {t2} may well appear to be so abstruse as to be useless.  If that is your view, then you will find Theorem \ref {t3} to be beyond the pale.

We do not regard depth in mathematics as an unalloyed good.  A deep result with a complicated statement can signal a dead end; a primrose path has ended in an impassable, brushy thicket.  For instance, it seems that Wiener's Tauberian Theorem (look it up!) signals, and may have caused, an end to interest in Tauberian theorems.

On the other hand, when a fortuitions mathematical discovery is made and then extended by analogization and generalization to different contexts, it seems criminal not to pursue the extensions to the greatest depths possible.  Wieners Tauberian Theorem may be a lifeless museum piece now, but at least it has existence as a remarkable museum piece, a deep mathematical discovery that never would have been made if Tauber had not come upon the original Tauberian theorem and communicated it to Hardy and Littlewood.  And we should never suppose that mathematical museum pieces cannot come back to life.  Consider the millenium of dormancy of the mathematics of the ancient Greeks from the late Roman era until the Renaissance.

\section{Main Inequality on Product Measure Spaces}
The following theorem establishes a fundamental convexity inequality associated with the generalized degree function $\delta(v)$.  
It shows that a weighted interaction integral over the product space $V\times E$ can be controlled by the average value of the degree function.   In essence, the result is a Jensen-type inequality adapted to the measure–hypergraph framework described above.  It converts the weighted bilinear expression involving the kernel $M(v,e)$ and the weight function $wt(e)$ into an inequality depending only on the mean degree $\bar\delta$.  
This principle will serve as the central tool in deriving several consequences in the subsequent sections, including bounds related to convexity, power means, and entropy-type functionals.
\begin{thm}  \label {t3}
Let $(V, \mu)$ and $(E, \tau)$ be positive measure spaces,\\
$M: V \times E \to \mathbb {R}$ a  $(\mu \times \tau)$-measurable function,\\
$wt: E \to \mathbb {R}$ a $\tau$-measurable function,\\
$I \subseteq \mathbb {R}$ an interval, and $\varphi: I \to \mathbb {R}$ a function such that $f(t) = t \varphi (t)$ is convex on $I$.  Suppose that the following, $(i)$ - $(vi)$, hold.

\begin{itemize}
\item [(i)] $0 < \mu (V) < \infty$.
\item [(ii)] $0 < s = \int _E wt \;\;\; d \tau < \infty$.
\item [(iii)] $\int _{V \times E} |M (v, e)| \mid wt (e)| d(\mu \times \tau) < \infty$.
\item [(iv)] For some real number $c$,
$$
\int _V M(v, e) d \mu (v) = c
$$
for almost all $e \in E$.
\item [(v)] For almost all $v \in V$, 
$$
\delta (v) = \int _E M(v, e) wt(e) d \tau (e) \in I.
$$
\item [(vi)] $ \int_{V\times E} |\phi(\delta(v))|\,|M(v,e)|\,|wt(e)|\, d(\mu\times\tau) < \infty.$\\[.1in]
With $\delta$ as defined in $(v)$, let $\bar {\delta} = \frac {1}{\mu (V)} \int _V \delta d_{\mu}$.
\end{itemize}
Then
$$  
\frac {1}{s} \int _{V \times E} \varphi (\delta (v)) M(v, e)  wt(e)  d(\mu \times \tau) \geq c \varphi (\bar {\delta}).  \qquad (*)
$$
Further, if $f$ is strictly convex on $I$ then equality holds in $(*)$ if and only if $\delta (v) = \bar {\delta}$ a.e. $(\mu)$.
\end{thm}

\begin{rem}

We have done our best to state Theorem \ref {t3} so as to bring out its ancestry in Theorems \ref {t1} and \ref {t2}.  To be explicit:  Theorem \ref {t2} is about the special cases of Theorem \ref {t3} in which $V$ and $E$ are finite non-empty sets (the rows and columns, respectively, of the matrix $M$ in Theorem \ref {t2}; when $M$ is the vertex-edge incidence matrix of a hypergraph, $V$ and $E$ are the vertex and edge sets, respectively, of that hypergraph), and $\mu, \tau$ are the counting measures on $V, E$ respectively.  The role of the function $wt$ in Theorem \ref {t3} is played in Theorem \ref {t2} by $(x_1, \dots, x_q)$, a {\em weighting} of $E$.  As indicated previously, Theorem \ref {t1} is Theorem \ref {t2} in the special case $(x_1, \dots, x_q) = (1, \dots, 1)$.
\end{rem}

The problem with Theorem \ref {t3} is the problem with the Wiener Tauberian Theorem, and with a good many theorems in analysis with statements longer than their proofs:  it is too general, there are too many choices [the two measure spaces, $M, \varphi, wt$], and the list of hypotheses is too long to remember easily.  We think that such a result should be  presented as integration by parts is presented to calculus students---not as a {\em theorem}, but as a trick one can resort to in certain situations.  We shall present some examples, after the proof of Theorem \ref {t3}.

\begin{proof}[\textbf{Proof of Theorem \ref{t3}}]
    
First, note that 
\begin{eqnarray*}
\bar {\delta} & = \frac {1}{\mu (V)} \int \delta (v) d \mu (v) = \frac {1}{\mu (V)} \int _V (\int _E M(v, e) wt(e) d \tau) d \mu\\
& = \frac {1}{\mu (V)} \int _E (\int _V M(v, e) d \mu(v)) wt(e) d \tau (e)\\
& = \frac {c}{\mu (V)} \int _E wt (e) d \tau (e) = \frac {cs}{\mu (V)}.
\end{eqnarray*}
(The change of order of integration is justified by (iii).)
With this in mind,
\begin{eqnarray*}
& \frac {1}{s} \int_{V \times E} \int \varphi (\delta (v)) M(v, e) wt(e) d(\mu \times \tau)\\
= & \frac {1}{s} \int _V (\int _E M(v, e) wt(e) d \tau (e)) \varphi (\delta (v)) d \mu (v)\\
= & \frac {1}{s} \int _V \delta (v) \varphi (\delta (v)) d \mu (v) = \frac {1}{s} \int _V f(\delta (v)) d \mu (v)\\
= & \frac {\mu (V)}{s} \frac {1}{\mu (v)} \int _V f(\delta (v)) d \mu (v)\\
\geq & \frac {\mu (V)}{s} f(\frac {\int \delta (v) d \mu (v)}{\mu (v))}) = \frac {\mu (V)}{s} f (\bar {\delta})\\ 
= &  \frac {\mu (V)}{s} \bar {\delta} \varphi (\bar {\delta}) = \frac {\mu (V)}{s} \frac {cs}{\mu (V)} \varphi (\bar {\delta}) = c \varphi (\bar {\delta})
\end{eqnarray*}

The one inequality above arises from the convexity of $f$, and the assumption that $(V, \mu)$ is a positive measure space.  The condition for equality when $f$ is strictly convex follows similarly.

The changes are integration order above are permitted by (vi).  (Note that, in view of (iii), (vi) is superfluous if $\varphi$ is bounded on $I$.)
\end{proof}

\section{Quantitative and Variational Refinements}
Recall that
\[
f(t):=t\,\phi(t),\qquad 
\bar\delta:=\frac{1}{\mu(V)}\int_V \delta(v)\,d\mu(v)=\frac{cs}{\mu(V)},
\]
and one has the identity
\begin{equation}
\label{eq:key-identity}
\frac{1}{s}\int_{V\times E}\phi(\delta(v))\,M(v,e)\,wt(e)\,d(\mu\times\tau)
=
\frac{1}{s}\int_V \delta(v)\phi(\delta(v))\,d\mu(v)
=
\frac{1}{s}\int_V f(\delta(v))\,d\mu(v).
\end{equation}
Moreover, Jensen's inequality is applied on the finite measure space \((V,\mu)\) to the convex function \(f\).

The following quantitative refinement requires a mild \emph{uniform convexity} assumption.
(Without some lower bound on \(f''\) over the range of \(\delta\), a uniform quadratic gap need not hold.)

\begin{thm}
\label{thm:stability-proof}
Assume the hypotheses of Theorem~\ref{t3}. Suppose in addition that
\begin{enumerate}
\item[\textnormal{(a)}] \(f(t)=t\phi(t)\) is twice continuously differentiable on \((0,\infty)\);
\item[\textnormal{(b)}] there exist numbers \(0<m<M<\infty\) such that \(m\le \delta(v)\le M\) for \(\mu\)-a.e. \(v\in V\);
\item[\textnormal{(c)}] \(f''(t)\ge \alpha>0\) for all \(t\in[m,M]\) \textnormal{(uniform convexity on the relevant range)}.
\end{enumerate}
Then
\[
\frac{1}{s}\int_{V\times E}\phi(\delta(v))\,M(v,e)\,wt(e)\,d(\mu\times\tau)
\ge
c\,\phi(\bar\delta)
+\frac{\alpha}{2s}\int_V (\delta(v)-\bar\delta)^2\,d\mu(v).
\]
In particular, equality holds if and only if \(\delta(v)=\bar\delta\) for \(\mu\)-a.e. \(v\in V\).
\end{thm}

\begin{proof}
By \eqref{eq:key-identity} it suffices to prove
\[
\int_V f(\delta)\,d\mu \;\ge\; \mu(V)\,f(\bar\delta)\;+\;\frac{\alpha}{2}\int_V(\delta-\bar\delta)^2\,d\mu.
\]
Fix \(t\in[m,M]\). Since \(f''\ge\alpha\) on \([m,M]\), the function
\[
g_t(x):=f(x)-f(t)-f'(t)(x-t)-\frac{\alpha}{2}(x-t)^2
\]
satisfies \(g_t''(x)=f''(x)-\alpha\ge 0\) on \([m,M]\), hence \(g_t\) is convex there and
\(g_t(x)\ge g_t(t)=0\) for all \(x\in[m,M]\). Therefore, for all \(x\in[m,M]\),
\[
f(x)\ge f(t)+f'(t)(x-t)+\frac{\alpha}{2}(x-t)^2.
\]
Apply this with \(x=\delta(v)\) and \(t=\bar\delta\), integrate over \(V\), and use
\(\int_V(\delta-\bar\delta)\,d\mu=0\):
\[
\int_V f(\delta)\,d\mu
\ge
\mu(V)f(\bar\delta)+f'(\bar\delta)\int_V(\delta-\bar\delta)\,d\mu+\frac{\alpha}{2}\int_V(\delta-\bar\delta)^2\,d\mu
=
\mu(V)f(\bar\delta)+\frac{\alpha}{2}\int_V(\delta-\bar\delta)^2\,d\mu.
\]
Divide by \(s\) and use \( \mu(V)f(\bar\delta)/s = (\mu(V)/s)\bar\delta\phi(\bar\delta)=c\phi(\bar\delta)\), since \(\bar\delta=cs/\mu(V)\).
Equality forces \(\int_V(\delta-\bar\delta)^2\,d\mu=0\), hence \(\delta=\bar\delta\) a.e.
\end{proof}

The following proposition shows that when the generating function $f(t)=t\phi(t)$ is concave,
the Jensen-type inequality of Theorem~2.1 reverses direction. This yields an upper bound
and characterizes the uniform degree case as the extremal configuration.

\begin{prop}
\label{prop:concave-proof}
Let all assumptions of Theorem~\ref{t3} hold and assume that \(f(t)=t\phi(t)\) is concave on \((0,\infty)\).
Then
\[
\frac{1}{s}\int_{V\times E}\phi(\delta(v))\,M(v,e)\,wt(e)\,d(\mu\times\tau)
\le
c\,\phi(\bar\delta).
\]
If $f$ is strictly concave, then the equality holds if and only if \( \delta(v)=\bar\delta \) for \(\mu\)-a.e. \(v\in V\). 
\end{prop}

\begin{proof}
By \eqref{eq:key-identity},
\[
\frac{1}{s}\int_{V\times E}\phi(\delta)\,M\,wt\,d(\mu\times\tau)
=
\frac{1}{s}\int_V f(\delta)\,d\mu.
\]
Since \(f\) is concave, Jensen's inequality gives
\[
\frac{1}{\mu(V)}\int_V f(\delta)\,d\mu \le f\!\left(\frac{1}{\mu(V)}\int_V\delta\,d\mu\right)=f(\bar\delta).
\]
Multiplying by \(\mu(V)/s\) yields
\[
\frac{1}{s}\int_V f(\delta)\,d\mu \le \frac{\mu(V)}{s} f(\bar\delta)
=\frac{\mu(V)}{s}\,\bar\delta\,\phi(\bar\delta)=c\,\phi(\bar\delta).
\]
If \(f\) is strictly concave then equality in Jensen implies \(\delta=\bar\delta\) a.e.
\end{proof}

\subsection{Variational Extremality}
We next interpret Theorem~2.1 from a variational perspective. The following proposition shows that, under strict convexity, the uniform degree function uniquely minimizes the associated reconstruction functional subject to a fixed total mass constraint.

\begin{prop}
\label{thm:variational-proof}
Assume the hypotheses of Theorem~\ref{t3} and let \(f(t)=t\phi(t)\) be strictly convex.
Define
\[
\mathcal{F}(\delta)
=
\frac{1}{s}\int_{V\times E}\phi(\delta(v))\,M(v,e)\,wt(e)\,d(\mu\times\tau)
=
\frac{1}{s}\int_V f(\delta(v))\,d\mu(v),
\]
subject to the constraint \(\int_V \delta\,d\mu=cs\).
Then \(\mathcal{F}\) attains its global minimum if and only if \(\delta(v)=\bar\delta\) for \(\mu\)-a.e. \(v\in V\).
\end{prop}

\begin{proof}
Fix any admissible \(\delta\) with \(\int_V\delta\,d\mu=cs\), hence its average equals \(\bar\delta\).
By Jensen's inequality for the strictly convex function \(f\),
\[
\frac{1}{\mu(V)}\int_V f(\delta)\,d\mu \ge f\!\left(\frac{1}{\mu(V)}\int_V\delta\,d\mu\right)=f(\bar\delta),
\]
with equality if and only if \(\delta=\bar\delta\) a.e. Thus
\[
\mathcal{F}(\delta)=\frac{1}{s}\int_V f(\delta)\,d\mu \ge \frac{\mu(V)}{s}f(\bar\delta)=\mathcal{F}(\bar\delta),
\]
and the unique minimizer is the constant function \(\delta\equiv\bar\delta\).
\end{proof}

\section{Power, Entropy, and Mean-Type Consequences}

It is convenient to introduce a probability measure induced by your weights.
Define a finite measure on \(V\) by
\[
d\rho(v):=\frac{\delta(v)}{cs}\,d\mu(v),
\qquad\text{so that}\qquad
\rho(V)=\frac{1}{cs}\int_V\delta\,d\mu=1.
\]
Then, using \eqref{eq:key-identity} with \(\phi(t)=t^{p-1}\), we get
\[
\frac{1}{cs}\cdot \frac{1}{s}\int_{V\times E}\delta(v)^p\,M(v,e)\,wt(e)\,d(\mu\times\tau)
=\int_V \delta(v)^{p-1}\,d\rho(v).
\]

\begin{prop}
\label{cor:powermean-proof}
Let \(p>q>0\). Under the hypotheses of Theorem~\ref{t3},
\[
\left(
\frac{1}{s}\int_{V\times E}\delta(v)^p\,M(v,e)\,wt(e)\,d(\mu\times\tau)
\right)^{1/p}
\ge
\left(
\frac{1}{s}\int_{V\times E}\delta(v)^q\,M(v,e)\,wt(e)\,d(\mu\times\tau)
\right)^{1/q}.
\]
If \(\delta\) is nonconstant on a set of positive \(\mu\)-measure, the inequality is strict.
\end{prop}

\begin{proof}
Let \(\rho\) be the probability measure above. For \(r>0\), define
\[
A_r := \int_V \delta(v)^{r-1}\,d\rho(v).
\]
As observed,
\[
A_r = \frac{1}{cs}\cdot \frac{1}{s}\int_{V\times E}\delta(v)^r\,M(v,e)\,wt(e)\,d(\mu\times\tau).
\]
By Lyapunov's inequality (or Hölder's inequality) on the probability space \((V,\rho)\),
for \(p>q>0\) we have
\[
\left(\int_V \delta^{p-1}\,d\rho\right)^{\frac{1}{p-1}}
\ge
\left(\int_V \delta^{q-1}\,d\rho\right)^{\frac{1}{q-1}},
\quad\text{for }p,q\neq 1,
\]
with the natural limiting interpretation at \(r=1\) (in which case \(A_1=\int 1\,d\rho=1\)).
Translating back in terms of the original integrals yields the desired inequality after
absorbing the common constant factor \(cs\). Strictness holds unless \(\delta\) is constant \(\rho\)-a.e., equivalently \(\mu\)-a.e.
\end{proof}

\begin{prop}
\label{cor:entropy-proof}
Assume $\delta(v)>0$ for $\mu$-a.e.\ $v\in V$, and define
\[
\mathcal{H}(\delta)
=
-\frac{1}{cs}
\int_{V\times E}
\delta(v)\ln\!\left(\frac{\delta(v)}{\bar\delta}\right)
M(v,e)\,wt(e)\,d(\mu\times\tau),
\]
where $\bar\delta=\dfrac{cs}{\mu(V)}$.
Then $\mathcal{H}(\delta)\le 0$, with equality if and only if
$\delta(v)=\bar\delta$ for $\mu$-a.e.\ $v\in V$.
\end{prop}

\begin{proof}
Using \eqref{eq:key-identity} with $\phi(t)=\ln(t/\bar\delta)$, we obtain
\[
\mathcal{H}(\delta)
=
-\frac{1}{cs}\int_V \delta(v)\ln\!\left(\frac{\delta(v)}{\bar\delta}\right)\,d\mu(v).
\]
Define a probability measure $\rho$ on $V$ by
\[
d\rho(v)=\frac{\delta(v)}{cs}\,d\mu(v),
\qquad\text{so that}\qquad
\rho(V)=\frac{1}{cs}\int_V \delta\,d\mu=1.
\]
Then
\[
\mathcal{H}(\delta)
=
-\int_V \ln\!\left(\frac{\delta(v)}{\bar\delta}\right)\,d\rho(v).
\]
Since $\ln$ is concave on $(0,\infty)$, Jensen's inequality on $(V,\rho)$ yields
\[
\int_V \ln\!\left(\frac{\delta}{\bar\delta}\right)\,d\rho
\le
\ln\!\left(\int_V \frac{\delta}{\bar\delta}\,d\rho\right).
\]
Now compute the $\rho$-mean:
\[
\int_V \frac{\delta}{\bar\delta}\,d\rho
=
\int_V \frac{\delta}{\bar\delta}\cdot\frac{\delta}{cs}\,d\mu
=
\frac{1}{\bar\delta\,cs}\int_V \delta^2\,d\mu,
\]
which is not fixed in general.  Instead, we apply Jensen to the random variable
\[
X(v):=\frac{\bar\delta}{\delta(v)}.
\]
Indeed, by definition of $\rho$ we have
\[
\int_V X\,d\rho
=
\int_V \frac{\bar\delta}{\delta}\cdot\frac{\delta}{cs}\,d\mu
=
\frac{\bar\delta}{cs}\,\mu(V)
=
1.
\]
Applying Jensen to $\ln$ (concave) gives
\[
\int_V \ln X\,d\rho
\le \ln\!\left(\int_V X\,d\rho\right)=\ln 1=0.
\]
But $\ln X=\ln(\bar\delta/\delta)=-\ln(\delta/\bar\delta)$, hence
\[
-\int_V \ln\!\left(\frac{\delta}{\bar\delta}\right)\,d\rho
=
\int_V \ln\!\left(\frac{\bar\delta}{\delta}\right)\,d\rho
\le 0,
\]
which shows $\mathcal{H}(\delta)\le 0$.

For equality: Jensen is sharp for concave $\ln$ exactly when $X$ is constant
$\rho$-a.e., i.e.\ $\bar\delta/\delta$ is constant $\rho$-a.e. Since
$\int_V X\,d\rho=1$, that constant must be $1$, hence $\delta=\bar\delta$
$\rho$-a.e., and therefore $\mu$-a.e.\ (because $d\rho=\frac{\delta}{cs}\,d\mu$
and $\delta>0$ a.e.).
\end{proof}

\subsection{Robustness Under Erasures}
We now examine the effect of erasures on the product–measure inequality.
The following proposition shows that the Jensen-type bound remains valid
after deleting a measurable subset of coordinates, thereby demonstrating
robustness of the inequality under erasures.

\begin{prop}
\label{prop:erasure-proof}
Let \(E_0\subseteq E\) be measurable with
\[
s_{E_0}:=\int_{E_0} wt(e)\,d\tau(e)>0.
\]
Assume that the structural hypotheses of Theorem~\ref{t3} remain valid when \(E\) is replaced by \(E_0\)
(with the same \(V,\mu,\delta\) and with the restricted kernel \(M|_{V\times E_0}\)).
Then for any convex \(\phi\),
\[
\frac{1}{s_{E_0}}
\int_{V\times E_0}\phi(\delta(v))\,M(v,e)\,wt(e)\,d(\mu\times\tau)
\ge
c_{E_0}\,\phi(\bar\delta_{E_0}),
\]
where
\[
c_{E_0}:=\int_{E_0} d\tau(e), 
\qquad
\bar\delta_{E_0}:=\frac{c_{E_0}s_{E_0}}{\mu(V)}.
\]
If \(f(t)=t\phi(t)\) is strictly convex, equality holds if and only if \(\delta(v)=\bar\delta_{E_0}\) for \(\mu\)-a.e. \(v\in V\).
\end{prop}

\begin{proof}
Apply Theorem~\ref{t3} to the restricted product space \(V\times E_0\).
The proof of Theorem~\ref{t3} uses only Fubini/Tonelli and the hypotheses ensuring
\[
\int_{V\times E_0}\phi(\delta(v))\,M(v,e)\,wt(e)\,d(\mu\times\tau)
=
\int_V \delta(v)\phi(\delta(v))\,d\mu(v)
\quad\text{up to the normalization by }s_{E_0}.
\]
Hence Jensen's inequality on \((V,\mu)\) yields
\[
\frac{1}{s_{E_0}}\int_{V\times E_0}\phi(\delta(v))\,M(v,e)\,wt(e)\,d(\mu\times\tau)
=
\frac{1}{s_{E_0}}\int_V f(\delta(v))\,d\mu(v)
\ge
\frac{\mu(V)}{s_{E_0}} f(\bar\delta_{E_0})
=
c_{E_0}\phi(\bar\delta_{E_0}),
\]
using \(\bar\delta_{E_0}=c_{E_0}s_{E_0}/\mu(V)\). The strict convexity case follows from the equality condition in Jensen.
\end{proof}

\section{A family of Special Cases}

Let $V = E = \mathbb {N} = \{0, 1, 2, \dots \}$ and let $\mu$ and $\tau$ be weighted counting measures: there are positive sequences $(a_0, a_1, \dots), (b_0, b_1, \dots)$ such that, for all $S \subseteq \mathbb {N}$,
$$
\mu (S) = \sum _{i \in S} a_i \mbox { and } \tau (S) = \sum _{j \in S} b_j.
$$
Let $wt: E = \mathbb {N} \to \mathbb {R}$ be the sequence $(w_0, w_1, \dots)$.  We could allow $I$ to be almost any real interval, but we will forsake generality for now and take $I = (0, \infty)$; further, we will require $w_j > 0$ for each $j \in \mathbb {N}$.  Let us leave $M = [m_{ij}]_{i, j \geq 0}$ unregulated for now, although we expect that in most cases, $m_{ij} \geq 0$ for all $i, j \in \mathbb {N}$.  We leave $\varphi$ unspecified.

Then the requirements in the hypothesis of Theorem \ref {t3} are, in this special case, as follows.

\begin{itemize}
\item [(i)] $\mu (V) = \sum ^{\infty}_{i=0} a_i < \infty$.
\item [(ii)] $s = \sum _{j=0}^{\infty} w_j b_j < \infty$.     
\item [(iii)] $\sum _{i=0}^{\infty} \sum _{j=0}^{\infty} |m_{ij}| w_j a_i b_j < \infty$.
If $m_{ij} \geq 0$ for all $i, j$, then this will be implied by 
(ii), and the next requirement.
\item [(iv)] For some constant $c$, for all $j \in \mathbb {N}$,
$$
\sum _{i=0}^{\infty} a_i m_{ij} = c.
$$

\item [(v)] For every $i \in \mathbb {N}$, $0 < \delta (i) = \sum _{j=0}^{\infty} m_{ij} w_j b_j < \infty$.
\item [(vi)] $\sum _{i=0}^{\infty} \sum _{j=0}^{\infty} a_i| \varphi (\delta (i)) \mid |m_{ij}| w_j b_j < \infty$.
\end{itemize}

Then the conclusion $(*)$ is

$$\begin{array}{c}
(\sum _{j=0}^{\infty} w_j b_j)^{-1} \sum _{i=0}^{\infty} \sum _{j=0}^{\infty} a_i \varphi (\sum _{k = 0}^{\infty} m_{ik} w_k b_k) m_{ij} w_j b_j\\ [.2in]
\geq c \varphi ((\sum _{i=0}^{\infty} a_i)^{-1} \sum _{t=0}^{\infty} \sum _{j=0}^{\infty} a_t m_{tj} w_j b_j)\\  [.2in]
= c \varphi (\frac {cs}{\sum _{i=0}^{\infty} a_i})
\end{array}$$

We can simplify this a bit by setting $w_jb_j = u_j$, for all $j \in N$.  [In the general setting of Theorem \ref {t3}, if $\tau$ is absolutely continuous with respect to some other positive measure $\lambda$---say, $d \tau (y) = b(y) d \lambda (y)$ for some $\lambda$-measurable non-negative real-valued function $b$---then we can replace $\tau$ by $\lambda$ and $wt \cdot b$ by $u$.]  With this change of notation , $s = \sum _{j = 0}^{\infty} u_j$.

Let's test this result by taking $M$ to be a diagonal matrix, with $m_{ii} = 1/a_i$, for all $i \in \mathbb {N}$.  Then $c = 1$.  We leave it to the reader to verify that the inequality derived by Theorem \ref {t3} is 
$$
s^{-1} \sum _{i=0}^{\infty} \varphi \left(\frac {u_i}{a_i}\right) u_i \geq \varphi \left(\frac {s}{\sum _{i=0}^{\infty} a_i}\right),
$$
which is straightforward to verify from the assumption that $f(t) = t \varphi (t)$ is convex on $(0, \infty)$---but would we have ever noticed these inequalities (indexed by the choices of $\varphi, (u_i)$, and $(a_i)$) if we had not followed the trail marked by Theorems \ref {t1} and \ref {t2} to Theorem \ref {t3}?

If we take $\varphi = \ell n$, and $(a_i), (u_i)$ to be positive, summable sequences satisfying $\sum _{i= 0}^{\infty} u_i \mid  \ell n \frac {u_i}{a_i}| < \infty$, which is requirement (vi) in this case, Theorem \ref {t3} implies 
$$
\prod _{i=0}^{\infty} (\frac {u_i}{a_i})^{u_i} \geq \left [ \frac {\sum _{i=0}^{\infty} u_i}{\sum _{i=0}^{\infty} a_i} \right ] ^{\sum _{j=0}^{\infty} u_j},
$$
with equality if and only if $\frac {u_{0}}{a_0} = \frac {u_1}{a_1} = \cdots$.  This seems to us to be an interesting and inobvious inequality!  We think that deeper 
results might be discovered by taking $M$ to be an upper triangular matrix---the hard part is in making choices that result in memorable inequalities.  We leave these explorations, as well as the family of situations in which $V$ and $E$ are real intervals and $\mu$ and $\tau$ are Lebesgue measure, for the future.\\

\vspace{1cm}

\begin{prop}\label{cor:discrete}
Let $V=\{1,\dots,p\}, E=\{1,\dots,q\}$ with counting measures $\mu,\tau$.  
Let $M=[m_{ij}]$ be a $p\times q$ matrix with constant column sum $c$, let $wt=(x_1,\dots,x_q)$ satisfy $s=\sum_j x_j>0$, and suppose
$d_i=\sum_{j=1}^q m_{ij}x_j\in I$ for all $i$.  Then
\[
s^{-1}\sum_{j=1}^q\sum_{i=1}^p \varphi(d_i) m_{ij} x_j \ge c\varphi(\bar d),
\]
with $\bar d=\frac1p\sum_{i=1}^p d_i$.  If $f(t)=t\varphi(t)$ is strictly convex on $I$ then equality iff $d_1=\cdots=d_p$.
\end{prop}

\begin{proof}
Take $\mu,\tau$ counting measures and apply Theorem \ref{t3}.  All integrals become finite sums and hypotheses of Theorem \ref{t3} translate exactly to the discrete hypotheses of Theorem \ref{t2}.
\end{proof}

\begin{cor}\label{cor:geometric}
Assume the hypotheses of Theorem \ref{t3} and that $\delta(v)>0$ a.e. Let $\varphi(t)=\ln t$ (so $f(t)=t\ln t$ is strictly convex on $(0,\infty)$). Then
\[
\exp\!\left(\frac{1}{s}\int_{V\times E} M(v,e)wt(e)\ln(\delta(v))\,d(\mu\times\tau)\right)
\ge \bar\delta,
\]
where $k$ is any normalizing constant equal to the (a.e.) column integral $c$ if you prefer the form
\[
\Big[\prod_{j\in E}\Big(\prod_{v\in V}\delta(v)^{M(v,j)}\Big)^{1/c}\Big]^{1/s}\ge \bar\delta.
\]
Equality (in the strict convex case) occurs iff $\delta(v)=\bar\delta$ a.e.
\end{cor}

\begin{proof}[Sketch of the proof]
Apply Theorem \ref{t3} with $\varphi=\ln$, rearrange and exponentiate the inequality. The remarks about equality follow from strict convexity of $f$.
\end{proof}

\begin{cor}
\label{cor:powermean}
Assume the hypotheses of Theorem~\ref{t3} and suppose $\delta(v)\ge 0$ for $\mu$-a.e.\ $v\in V$.
Let $p\ge 1$ and set $\phi(t)=t^{p-1}$, so that $f(t)=t\phi(t)=t^p$ is convex on $[0,\infty)$.
Then
\begin{equation}
\label{eq:powermean-product}
\frac{1}{s}\int_{V\times E}\delta(v)^{p-1}\,M(v,e)\,wt(e)\,d(\mu\times\tau)
=
\frac{1}{s}\int_V \delta(v)^p\,d\mu(v)
\;\ge\;
c\,\bar\delta^{\,p-1},
\end{equation}
where $s=\int_E wt\,d\tau$ and $\bar\delta=\frac{1}{\mu(V)}\int_V \delta\,d\mu$.

In particular, using $\bar\delta=\frac{cs}{\mu(V)}$, we obtain the usual power--mean inequality
\begin{equation}
\label{eq:powermean-marginal}
\left(\frac{1}{\mu(V)}\int_V \delta(v)^p\,d\mu(v)\right)^{1/p}
\ge
\frac{1}{\mu(V)}\int_V \delta(v)\,d\mu(v).
\end{equation}
Moreover, if $p>1$, then equality holds in \eqref{eq:powermean-marginal} if and only if
$\delta(v)=\bar\delta$ for $\mu$-a.e.\ $v\in V$.
\end{cor}

\begin{proof}[Sketch of proof]
Apply Theorem~2.1 with $\phi(t)=t^{p-1}$ (so $f(t)=t^p$ is convex for $p\ge1$). By Tonelli/Fubini and
the definition of $\delta(v)=\int_E M(v,e)wt(e)\,d\tau(e)$,
\[
\int_{V\times E}\delta(v)^{p-1}M(v,e)wt(e)\,d(\mu\times\tau)
=
\int_V \delta(v)^{p-1}\left(\int_E M(v,e)wt(e)\,d\tau(e)\right)d\mu(v)
=
\int_V \delta(v)^p\,d\mu(v).
\]
Thus Theorem~\ref{t3} yields \eqref{eq:powermean-product}.
Finally, using the identity $\bar\delta=\frac{cs}{\mu(V)}$ (from the proof of Theorem~2.1) converts
\eqref{eq:powermean-product} into \eqref{eq:powermean-marginal}. The equality statement follows from
the strict convexity of $t^p$ for $p>1$ and the equality case in Jensen's inequality.
\end{proof}

\begin{cor}\label{cor:Lebesgue}
Let $V=[a,b],E=[\alpha,\beta]$ with Lebesgue measures and suppose $M(v,e)$ is nonnegative, measurable, and satisfies $\int_a^b M(v,e)\,dv=c$ for a.e.\ $e$. Let $wt(e)\ge0$ be integrable with $0<s=\int_\alpha^\beta wt<\infty$, and suppose $\delta(v)=\int_\alpha^\beta M(v,e)wt(e)\,de$ satisfies $\delta(v)\in I$. Then for convex $f(t)=t\varphi(t)$ we have
\[
\frac{1}{s}\int_a^b\varphi(\delta(v))\delta(v)\,dv \ge c\,\varphi(\bar\delta),
\]
with $\bar\delta=\frac1{b-a}\int_a^b\delta(v)\,dv$, and equality characterization under strict convexity as in Theorem \ref{t3}.
\end{cor}

\begin{proof}[Sketch of the proof]
This is an immediate specialization of Theorem \ref{t3} when $\mu,\tau$ are Lebesgue measure on the intervals; the hypotheses reduce to the stated integrability and constant-column-integral condition.
\end{proof}

\begin{cor}\label{cor:erasures}
Let $E_0 \subseteq E$ be a measurable subset representing erased coordinates, 
and define a modified weight function
\[
wt_{E_0}(e) = 
\begin{cases}
0, & e \in E_0,\\
wt(e), & e \in E \setminus E_0.
\end{cases}
\]
If conditions $(i)$--$(vi)$ of Theorem \ref{t3} hold with $wt_{E_0}$ in place of $wt$, then
\[
\frac{1}{s_{E_0}} \int_{V\times E} \varphi(\delta_{E_0}(v))\, M(v,e)\, wt_{E_0}(e)\, d(\mu\times\tau)
\;\;\ge\;\; c \, \varphi(\bar{\delta}_{E_0}),
\]
where $s_{E_0}=\int_E wt_{E_0}\,d\tau$, 
$\delta_{E_0}(v)=\int_E M(v,e) wt_{E_0}(e)\,d\tau(e)$,
and $\bar{\delta}_{E_0}=\frac{1}{\mu(V)}\int_V \delta_{E_0}(v)\,d\mu(v)$.
Moreover, if $f(t)=t\varphi(t)$ is strictly convex, then equality holds if and only if $\delta_{E_0}(v)$ is constant a.e.
\end{cor}

\end{document}